\documentclass[10pt]{amsart}
\usepackage{amsfonts}
\usepackage{mathrsfs}
\usepackage{amsmath,amsfonts,amssymb,amsthm,comment}
\usepackage{amscd}

\title{Mirkovic-Vilonen cycles and polytopes for
a symmetric pair}
\author{Jiuzu Hong }

\address{Address: Academy of Mathematics and Systems Science, Chinese Academy of Sciences, Beijing 100080, China}
\address{Current Address: School of Mathematical Sciences, Tel Aviv University, Ramat Aviv, Tel Aviv 69978, Israel.}
\email{hjzzjh@gmail.com}

\newtheorem{thm}{Theorem}[section]
\newtheorem{specialthm}{Theorem}
\newtheorem{cor}[thm]{Corollary}
\newtheorem{lem}[thm]{Lemma}
\newtheorem{prop}[thm]{Proposition}

\theoremstyle{definition}

\theoremstyle{remark}
\newtheorem{rem}{Remark}[section]

\newcommand{\C}{\mathbb{C}}
\newcommand{\G}{\mathcal{G}}
\newcommand{\K}{\mathcal{K}}
\newcommand{\st}{\sigma}
\newcommand{\R}{\mathbb{R}}
\newcommand{\N}{\mathbb{N}}
\newcommand{\Z}{\mathbb{Z}}

\newcommand{\D}{\cdot}

\begin{document}

\maketitle

\markboth {Jiuzu Hong}{Mirkovic-Vilonen cycles and polytopes for a
Symmetric pair}
\begin{abstract}
Let $G$ be a connected, simply-connected, and almost simple
algebraic group, and let $\st$ be a Dynkin automorphism on $G$. Then
$(G,G^\st)$ is a symmetric pair. In this paper, we get a bijection
between the set of  $\st$-invariant MV cycles (polytopes) for $G$
and the set of MV  cycles (polytopes) for $G^\st$, which is the
fixed point subgroup of $G$; moreover, this bijection can be
restricted to the set of MV cycles (polytopes) in irreducible
representations. As an application, we obtain a new proof of the
twining character formula.
\end{abstract}

\section{Introduction}

Let $G$ be a connected semisimple algebraic group over $\C$, and let
$\G$ be the affine Grassmannian of $G$. Let $\G_\lambda$ be the
$G(\C[[t]])$-orbit on $\G$ corresponding to a dominant coweight
$\lambda$ on $G$. Let $IC_\lambda$ be the spherical perverse sheaf
supported on $\overline{\G_\lambda}$. V.Ginzburg \cite{G} and
Mirkovi\'{c}-Vilonen \cite{MV} set up the geometric Satake
correspondence, which says that the category of spherical perverse
sheaves on $\G$ is equivalent to the category of finite dimensional
representations of the Langlands dual group $G^\vee$ of $G$; in
particular, the irreducible representation $V(\lambda)$ of $G^\vee$
with highest weight $\lambda$ is identified with the cohomology
group $H^{*}(\G,IC_\lambda)$. Furthermore, Mirkovi\'{c} and Vilonen
\cite{MV} discovered Mirkovi\'{c}-Vilonen cycles which affords a
natural basis of $V(\lambda)$.

In \cite{A}, Anderson studied the moment polytopes of
Mirkovi\'{c}-Vilonen cycles, which are called Mirkovi\'{c}-Vilonen
polytopes, and showed that these polytopes could be used to
understand the combinatorics of representations of $G^\vee$. In
\cite{K1}, Kamnitzer gave an explicit combinatorial description of
the MV cycles and polytopes. He showed that canonical basis and MV
cycles are governed by the same combinatorics, i.e MV cycles
$\longleftrightarrow$ MV polytopes $\longleftrightarrow$ canonical
basis, are bijections.

Let $\sigma$ be a nontrivial Dynkin automorphism of $G$. We have a
Dynkin automorphism on $G^\vee$ induced from $\st$. Let $G^\st$ be
the identity component of fixed point group of $\st$ on $G$. Let
$\lambda$ be a $\sigma$-invariant dominant coweight of $G$, which
can also be viewed as a dominant coweight of $G^\st$. let
$V(\lambda)$ be the irreducible representation of $G^\vee$ with
highest weight $\lambda$. We have a natural action of $\sigma$ on
$V(\lambda)$ induced from the action of the automorphism on
$G^\vee$, which fixes the highest weight vector in $V(\lambda)$. For
a $\sigma$-invariant coweight $\mu$ for $G$, $\sigma$ acts on the
weight space $V_\mu(\lambda)$. The twining character is defined to
be $\sum_{\st(\mu)=\mu}
\text{\rm{trace}}(\st|_{V_\mu(\lambda)})e^\mu$. It is related to the
character of the irreducible representation of $(G^\st)^\vee$ with
highest weight $\lambda$ through the twining character formula,
which is attributed to Jantzen \cite{J} under the name of Jantzen
theorem in \cite{KLP}. Though there are many proofs in the
literature ( for example \cite{J}, \cite{N}, \cite{KLP} ), it seems
that there is no satisfactory explanation for why Langlands dual
appears in this formula.

In this paper, we consider the action of $\sigma$ on MV cycles and
MV polytopes. The main result of the paper is to give an explicit
bijection between $\st$-invariant MV cycles (polytopes) for $G$ to
MV cycles (polytopes) for $G^\st$. In terms of polytopes, it sends
$\st$-invariant MV polytopes $P$ for $G$, to $P^\st$, which is a MV
polytope for $G^\st$. The bijection can be restricted to MV cyles
(polytopes) in irreducible representation space.

In this paper, we also show that the automorphism on $G^\vee$ from
Tannakian formalism is a Dynkin automorphism. On $V(\lambda)$, there
are two actions of $\st$, where one is induced from $G^\vee$, and
the other one is induced from the action of $\st$ on MV cycles. We
show that both of them agree, then we get a new proof of twining
character formula through geometric
Satake correspondence. \\

I would like to thank Professor T.Tanisaki for his contributions to
the paper, including ideas, discussions and useful comments; I also
appreciate his careful reading and enormous help on the improvement
of writing. I am very indebted to Professor N.Xi for his support and
encouragement on my mathematics study, and also much help to the
paper. Finally, I would like to thank Professor G.Lusztig for some
beneficial conversations during his stay in China in July, 2007.

The paper was written during the author's visit to Hebrew University
in Jerusalem under the support of \textsf{Marie Curie Research
Training Network}. During the writing of this paper, I heard that
Professor S.Naito and D.Sagaki have  given a closely related result
almost at the same time \cite{NS1} \cite{NS2}.

I would like to thank the referee for very helpful comments.
\section{Dynkin automorphism}\label{section 2}
\subsection{Notations}\label{sectioin 2.1}
Let $G$ be a connected, simply-connected and almost simple algebraic
group of rank $\ell$ over $\C$. Let $T$ be a maximal torus of $G$
and let $X^*=\text{\rm{Hom}}(T,\C^\times)$,
$X_*=\text{\rm{Hom}}(\C^\times,T)$ denote the weight and coweight
lattices of $T$. Then we have a natural perfect pairing
$\langle,\rangle:X_* \times X^*\rightarrow \mathbb{Z}$. Let
$W=N(T)/T$ denote the Weyl group.

Let $I=\{1,\cdots,l\}$ denote vertices of the Dynkin diagram of $G$.
Let $B$ be a Borel subgroup of $G$ containing $T$. Let $\alpha_1,
\alpha_2,\cdots,\alpha_l$ and $\alpha_1 ^\vee, \alpha_2
^\vee,\cdots,\alpha_l^\vee$ be simple roots and simple coroots of
$G$ with respect to $B$, respectively. Then $a_{ij}=\langle \alpha_i
^\vee, \alpha_j\rangle$ is the entry of the Cartan matrix of $G$.
Note that $(X_*,X^*,\langle,\rangle,\alpha_i ^\vee,\alpha_i;i\in I)$
is the root datum of $G$. Let $\lambda_1,\cdots,\lambda_l$ $\in
X^*\otimes \mathbb{R}$ be fundamental weights.

For $i \in I$, let $x_i:\C\rightarrow G$ and $y_i:\C\rightarrow G$
be root homomorphisms (corresponding to $\alpha_i$ and $-\alpha_i$,
respectively) which together with $T$, $B$ form a pinning of $G$.

Let $s_1, \cdots, s_{\ell} \in W$ be the set of simple reflections.
Let $w_0$ be the longest element of $W$, and let $m$ be its length.

We use $\geq$ for the usual partial order on $X_*$, so that $\mu\geq
\nu$ if and only if $\mu-\nu$ is a sum of positive coroots. More
generally, for each $w\in W$, we have the twisted partial order
$\geq_w$ on $X_*$, where $\mu \geq_w \nu$ if and only if
$w^{-1}\D\mu \geq w^{-1}\D\nu$.

A reduced word for an element $w\in W$ is a sequence of indices
$\bold{i}=(i_1,\cdots,i_k)$ $\in I^k$ such that $w=s_{i_1}\D
s_{i_2}\cdots s_{i_k}$ is a reduced expression. In this paper, a
reduced word will always mean a reduced word for $w_0$, where $w_0$
is the longest element in $W$.
\subsection{Group structure of $G^\st$}\label{section 2.2}

Let $\sigma:I\to I$ be a nontrivial bijection, satisfying
$a_{\sigma(i)\sigma(j)}=a_{ij}$ for all $i, j\in I$. We assume that
there are automorphisms $\sigma:X^*\to X^*$ and $\sigma:X_*\to X_*$
of ${\bf Z}$-modules satisfying
$\sigma(\alpha_i)=\alpha_{\sigma(i)}$ and
$\sigma(\alpha^\vee_i)=\alpha^\vee_{\sigma(i)}$ for any $i\in I$.
Then $\sigma$ induces an automorphism $\sigma:G\to G$ of algebraic
groups, such that $\st(x_i(a))=x_{\st(i)}(a)$ and
$\st(y_i(a))=y_{\st(i)}(a)$ $(\forall\text{ } i\in I)$. We call
$\st$ a \textsf{Dynkin automorphism} on $G$. In particular, we have
$\st(B)=B$ and $\st(T)=T$.

Let $G^\st$ be the fixed point group of $\st$ on $G$, and let
$T^\st$ and $B^\st$ be the fixed point groups of $T$ and $B$,
respectively. Then $G^\st$, $B^\st$ and $T^\st$ are connected,
moreover $G^\st$ is almost simple algebraic group, under our
assumptions on $G$, see \cite{S}. We call $(G,G^\st)$ a symmetric
pair.

We set $X_* ^\st=\{\lambda \in X_*| \st(\lambda)=\lambda\}$, and
$X^*_\st=\text{\rm{Hom}}(X_* ^\st,\Z)$. We have a perfect pairing
$X_* ^\st \times X^*_\st \rightarrow \mathbb{Z}$ denoted again by
$\langle,\rangle$. Let $I_\st$ be the set of $\st$-orbits on $I$.

For any $\eta \in I_\st$, let $\alpha_\eta ^\vee=2^h\sum_{i\in \eta}
\alpha_i ^\vee  \in X_* ^\st$, where $h$ is the number of unordered
pairs $(i,j)$ such that $i,j\in \eta$, $\alpha_i +\alpha_j \in
\Phi$. Note that $h=1$, if $\eta=\{i,j\}$ and $a_{ij}=-1$; $h=0$,
otherwise. Let $\theta:X^*\otimes\R\rightarrow X^*_\st\otimes \R$ be
the natural surjection induced from the perfect pairing $\langle,
\rangle:X_* \times X^*\rightarrow \Z$. Set
$\alpha_\eta=\theta(\alpha_i)$, and $\lambda_\eta=\frac{1}{h}
\theta(\lambda_i)$, where $i$ is any element of $\eta$. We have the
following proposition (see \cite{KLP}, \cite{J}).
\begin{prop}
  $( X_* ^\st,X^*_\st, \alpha_\eta ^\vee, \alpha_\eta)$ is a root
  datum of $G^\st$.
\end{prop}

 Define $x_\eta = \prod_{i\in \eta} x_i: \C \rightarrow G^\st$, by
$x_\eta(a)=\prod_{i\in \eta} x_i(a)$, if $\eta$ has only one
element, or $\forall$ $ i,j \in \eta$, with $i\neq j$, $a_{ij}=0$;
define $x_\eta:\C\rightarrow G^\st$, by
$x_\eta(a)=x_i(a)x_{j}(2a)x_i(a)$, if $\eta=\{i,j\}$, $a_{ij}=-1$.
We have the following lemma, see \cite{L2}.
\begin{lem}
  Let $x_1$, $x_2$ be two simple root subgroup homomorphisms of $G$ of type $A_2$
  corresponding to $\alpha_1$ and $\alpha_2$. Then we have
  $x_1(a_1)x_2(a_2)x_1(a_3)=x_2(\frac{a_2a_3}{a_1+a_3})x_1(a_1+a_3)x_2(\frac{a_1a_2}{a_1+a_3})$.
\end{lem}

From this lemma, we see easily that $x_\eta$ is a group
homomorphism. Similarly, we can define $y_\eta$, so that $x_\eta$
and $y_\eta$ are homomorphisms from $\C$ to $G^\st$. Since $t
x_\eta(a)t^{-1}=x_\eta(\alpha_\eta(t)a)$, $x_\eta$ is a root
subgroup homomorphism of $G^\st$ with root $\alpha_\eta$. We have
\begin{prop}
 $(T^\st, B^\st, x_\eta,
y_\eta; \eta \in I_\st)$ form a pinning of $G^\st$.
\end{prop}

Clearly, $\st :G\rightarrow G$ induces an automorphism of $W$
denoted again by $\st$, satisfying $\st(s_i)=s_{\st(i)}$ for any
$i\in I$. Let $W^\st =\{w\in W| \st(w)=w\}$. For any $\eta \in
I_\st$ we define $s_\eta \in W^\st$ to be the longest element in the
subgroup of $W$ generated by $\{s_i; i\in \eta\}$. It is known that
$W^\st$ is a Coxeter group on the generators $\{s_\eta; \eta \in
I_\st\}$. Any element $w \in W^\st$ can be restricted to $X_* ^\st$.
Under this restriction, we can see that $W^\st$ is identified with
the Weyl group of $G^\st$. For $w\in W^\st$, we denote by
$\ell_\st(w)$ the length of $w$ in the Coxeter group $W^\st$.

\section{MV cycles and MV polytopes for the symmetric pair}\label{section 4}

\subsection{Action of $\st$ on Affine Grassmannian} \label{section 3}
Let $\mathcal{O}=\C[[t]]$, and let $\mathcal{K}$ be the quotient
field of $\mathcal{O}$. Let $\G$ and $\G_\st$ be affine Grassmannian
of $G$ and $G^\st$ respectively. As the sets of rational points over
$\C$, $\G=G(\K)/G(\mathcal{O})$, and
$\G_\st=G(\K)^\st/{G(\mathcal{O})^\st}$. A coweight $\mu \in X_*$
gives a point in $\G$, denoted by $\underline{t}^{\mu}$. It is know
that $\underline{t}^{\mu}$ is a fixed points for the action of $T$
on $\G$. In fact all the fixed points of $T$ are given in this way.

For a given dominant coweight $\lambda$, we set $\G
^\lambda=G(\mathcal{O})\cdot \underline{t}^{\lambda}$. We have the
decomposition $\G=\bigsqcup_{\lambda \in X_* ^+} \G ^\lambda$, where
$X_* ^+$ is the set of dominant coweights.

Let $N$ be the unipotent radical of $B$. For $w\in W$, we set
$N_w=wNw^{-1}$. For $w\in W$ and $\mu \in X_*$, define the
\textsf{semi-infinite cells} by $S_w ^\mu =N_w(\K)\cdot
\underline{t}^\mu$. For simplicity, we set $S^\mu=S_e
^\mu=N(\K)\cdot \underline{t}^\mu$. We have $\G=\bigsqcup_{\mu\in
X_*} S^\mu$. The semi-infinite cells have the simple containment
relation, $\overline{S_w ^\mu}=\bigsqcup_{\nu\leq_w \mu} S_w ^\nu$.
We see that if $S_w ^\mu \cap S_v ^\nu \neq \O$, then $\nu \leq_w
\mu$.

We have the closed embedding $\iota:\G_\st \hookrightarrow \G$.
Since $\st(S^\lambda)=S^{\st(\lambda)}$, we have
$\G^\st=\bigsqcup_{\lambda\in X_* ^\st}(S^\lambda)^\st$.

Set $U:=\{g(t^{-1})\in G(\C[t^{-1}])|g(0)=1)\}$. Then the fixed
point set $U^\st=\{g(t^{-1})\in G^\st(\C[t^{-1}])|g(0)=1)\}$. For a
coweight $\lambda$, set $S(\lambda):=N(\C[t,t^{-1}]) \cap t^\lambda
U t^{-\lambda}$ and $S_\st(\lambda):=N^\st(\C[t,t^{-1}]) \cap
t^\lambda U^\st t^{-\lambda}$.

The following result should be well-known.
\begin{lem}\label{sim}
  Let $\lambda \in X_*$. Then the group $S(\lambda)$ acts
  simply-transitively on $S ^\lambda$, i.e., $S(\lambda)\simeq S
  ^\lambda$, with the map $g \mapsto g.t^\lambda$.
\end{lem}

\begin{prop}\label{pro:fixed}
  The fixed point subvariety of the action of $\st$ on $\G$ is
  exactly identified with $\G_\st$.
\end{prop}

\begin{proof}

From Lemma \ref{sim}, we are reduced to show
$S(\lambda)^\st=S_\st(\lambda)$ for $\lambda \in X_* ^\st$, and it
is easy to see, since
$$S(\lambda)^\st=N(\C[t,t^{-1}])^\st \cap (t^\lambda U t^{-\lambda})^\st=N^\st(\C[t,t^{-1}]) \cap t^\lambda U^\st t^{-\lambda}=S_\st(\lambda).$$
\end{proof}

From $\overline{\G^\lambda}=\bigsqcup_{\mu \leq \lambda} \G ^\mu$,
$\overline{S_w ^\mu}=\bigsqcup_{\nu\leq_w \mu} S_w ^\nu$ and the
above proposition, we can easily see that
\begin{cor}\label{cor}
For $\lambda$ a $\st$-invariant, and $w$ a $\st$-invariant element
in W, we have $(\G ^\lambda)^\st=\G_\st ^\lambda$, $\overline{\G
^\lambda}^\st=\overline{\G_\st ^\lambda}$, $(S_w ^\mu)^\st=(S_\st)_w
^\mu$, and $\overline{S_w ^\mu}^\st=\overline{(S_\st)_w ^\mu}$.
\end{cor}

\subsection{MV cycles and MV polytopes}\label{section 4.1}

 Let $\mu_1$, $\mu_2$ be coweights with $\mu_1\geq \mu_2$.
Following Anderson \cite{A}, an irreducible component of
$\overline{S_e ^{\mu_1} \cap S_{w_0} ^{\mu_2}}$ is called an
\textsf{MV cycle} with coweight $(\mu_1, \mu_2)$. This definition of
an MV cycle is a generalization of the original one in \cite{MV}.
$X_*$ acts on $\G$ by $\nu\cdot L:=t^\nu\cdot L$. Since $T$
normalizes $N_w$, we see that $\nu\cdot S_w ^{\mu}=S_w ^{\mu+\nu}$.
If $A$ is a component of $\overline{S_e ^{\mu_1} \cap S_{w_0}
^{\mu_2}}$, then $\nu\cdot A$ is a component of $\overline{S_e
^{\mu_1+\nu} \cap S_{w_0} ^{\mu_2+\nu}}$. Hence $X_*$ acts on the
set of all MV cycles. The orbit of an MV cycle with coweight
$(\mu_1, \mu_2)$ is called a stable MV cycle with coweight $\mu_2
-\mu_1$. Note that a stable MV cycle with coweight $\mu$ has a
unique representative with coweight $(\nu, \nu+\mu)$ for a fixed
coweight $\nu$.

Let $\text{\rm{MVC}}_G$ denote the set of stable MV cycles for $G$,
and let $\text{\rm{MVC}}^\mu_G$ denote the set of those with
coweight $\mu$.  For a $T$-invariant closed subvariety $A$ of the
affine Grassmannian, let $\Phi(A)\subset
t_\mathbb{R}:=X_*\otimes\mathbb{R}$ be the moment polytope of $A$,
which is exactly the convex hull of $\{\mu \in X_*| t^\mu \in A\}$.

If $A$ is an MV cycle with coweight $(\mu_1,\mu_2)$, then we say
that $\Phi(A)$ is an \textsf{MV polytope} with coweight $(\mu_1,
\mu_2)$. The action of $X_*$ on the set of MV cycles gives an action
of $X_*$ on the set of MV polytopes. It is easy to see that
$\nu\cdot P=P+\nu$. The orbit of $X_*$ on an MV polytope with
coweight $(\mu_1,\mu_2)$ is called a stable MV polytope with
coweight $\mu_2-\mu_1$.

Let $\text{\rm{MVP}}_G$ be the set of stable MV polytopes for $G$,
and let $\text{\rm{MVP}}^\mu_G$ be the set of stable MV polytopes
for $G$ with coweight $\mu$. As mentioned in \cite{A}, there is a
natural bijection between $\text{\rm{MVC}}_G$ and
$\text{\rm{MVP}}_G$. Let $C$ be an MV cycle, and $[C]$ be its stable
MV cycle. Let $P_C$ be the corresponding MV polytope of $C$, and
$[P_C]$ be its stable MV polytope. If there is no confusion, we
write $C$ (resp. $P$) for both MV cycle (or polytope) and stable MV
cycle (resp. polytope).

Suppose we are given a collection of coweights
$\mu_\bullet=(\mu_w)_{w\in W}$ such that $\mu_v \leq_w \mu_w$ for
all $v,w \in W$. Then we define the corresponding pseudo-Weyl
polytope by:
\begin{equation*}\label{polytope}
P(\mu_\bullet):=\cap_w C_w ^{\mu_w}=\{\alpha| \langle\alpha,w\cdot
\lambda_i\rangle\leq \langle\mu_w, w\cdot \lambda\rangle,
\forall\text{ } w\in W,\text{ and } i\in I\}.
\end{equation*}

For a collection $(\mu_w)_{w\in W}$ with coweights such that
$\mu_y\leq_w \mu_w$, for any $y, w\in W$, set $A(\mu_\bullet)=\cap
S_w ^{\mu_w}$, and let $\textrm{Conv}(\mu_\bullet)$ be the convex
hull of $(\mu_w)_{w\in W}$ in $t_\R$. $A(\mu_\bullet)$ is called a
\emph{GGMS stratum}, and it is a candidate of MV cycles. If it is
not empty, then the moment polytope of $\overline{A(\mu_\bullet)}$
is exactly $\textrm{Conv}(\mu_\bullet)$ (see Lemma 2.3, \cite{K1}),
which also coincides with $P(\mu_\bullet)$. That is,
$\textrm{Conv}(\mu_\bullet)=P(\mu_\bullet)$.

The following theorem gives a criterion for the closure of a GGMS
stratum to be an MV cycle.
\begin{specialthm}[Kamnitzer\cite{K1}]
Let $(\mu_w)_{w\in W}$ be the set with coweights, such that
$\mu_y\leq_w \mu_w$, for any $y, w\in W$. Then
$\overline{A(\mu_\bullet)}=\overline{\cap S_w
 ^{\mu_w}}$ is an MV cycle if and only if $\text{\rm{Conv}}(\mu_\bullet)$ is an MV
  polytope.
\end{specialthm}

Let $P$ be an MV polytope with vertices $(\mu_w)_{w\in W}$. Then $P$
is the moment polytope of an MV cycle $\overline{\cap S_w
^{\mu_w}}$. In this case, $\st(\overline{\cap S_w
^{\mu_w}})=\overline{\cap S_w ^{\st(\mu_{\st^{-1} (w)})}}$ is also
an MV cycle, and its moment polytope is exactly
$\text{\rm{Conv}}(\st(\mu_{\st^{-1}(w)}))$. Hence it is an MV
polytope with vertices $(\st(\mu_{\st^{-1}(w)}))_{w\in W}$, which
coincides with $\st(P)$.

\begin{lem}\label{equiv}
Let $(\mu_w)_{w\in W}$ be the vertices of an MV polytope $P$, and
let $A(\mu_\bullet)$ be the corresponding GGMS stratum, such that
$\overline{A(\mu_\bullet)}$ is an MV cycle. Then the following
statements are equivalent:
\begin{enumerate}
\item
$P$ is $\st$-invariant.
\item
 $\overline{A(\mu_\bullet)}$ is $\st$-invariant.
\item
 $A(\mu_\bullet)$ is $\st$-invariant.
\item
 $\st(\mu_w)=\mu_{\st(w)}$, $\forall\text{ } w \in W$.
 \end{enumerate}
\end{lem}
\begin{proof}
Since MV cycles are parametrized by MV polytopes bijectively, it is
easy to see that the moment polytope of $\st(\overline{\cap
S_w^{\mu_w}})$ is $\st(P)$. So $P$ is $\st$-invariant if and only if
$\overline{A(\mu_\bullet)}$ is $\st$-invariant, i.e.,
$(1)\Leftrightarrow(2)$.

Assume $\overline{A(\mu_\bullet)}$ is $\st$-invariant. Then
$\overline{ \cap S_w ^{\mu_w}}= \overline{\cap S_w
^{\st(\mu_{\st^{-1} (w)})}}$. Since $\cap S_w ^{\mu_w}$ and $\cap
S_w ^{\st(\mu_{\st^{-1} (w)})}$ are locally closed, we have $(\cap
S_w ^{\mu_w}) \cap (\cap S_w ^{\st(\mu_{\st^{-1} (w)})}) \neq \O$.
It implies that, $\forall\text{ } w\in W$, $S_w ^{\mu_w} \cap S_w
^{\st(\mu_{\st^{-1} (w)})}\neq \O.$ Hence $\mu_w=\st(\mu_{\st^{-1}
(w)}), \forall\text{ } w\in W$. So $(2)\Rightarrow (4)$.

It is easy to see $(3)\Leftrightarrow (4)$, and (4) implies (1)
immediately.
\end{proof}

\subsection{Lusztig datum }\label{section 4.3}
Let $\bold{i}$ be a reduced word, and $n_\bullet \in \N ^m$. Recall
some results in \cite{K1}. We define $\{\mu_{w_k ^\bold{i}}\}_{0\leq
k\leq m}$ inductively by $\mu_e =0$ and $\mu_{w_k
^\bold{i}}=\mu_{w_{k-1}^\bold{i}}-n_k w_{k-1} ^\bold{i}(\alpha_{i_k}
^\vee)$, for any $1\leq k\leq m$. Set $A^\bold{i}(n_\bullet)=\cap
S_{w_k ^\bold{i}} ^{\mu_{w_k^\bold{i}}}$. Then
$\overline{A^\bold{i}(n_\bullet)}$ is an MV cycle with coweight
$\mu_{w_0}$, and the corresponding MV polytope $P$ has
$\bold{i}$-Lusztig datum $n_\bullet$. From the corresponding
$\bold{i}$-Lusztig datum of the MV polytope $P$, we can recover the
vertices of $P$ uniquely, through the above procedure. In this way,
we have a bijection from MV polytopes to $\bold{i}$-Lusztig data.
Moreover, there exists an explicit bijection between
$\bold{i}$-Lusztig data and MV cycles, $\tau_{\bold{i}}:
\N^m\rightarrow \text{\rm{MVC}}$ by $\tau_{\bold{i}}(n_\bullet)=
\overline{A^{\bold{i}}(n_\bullet)}$.

Let $\bold{i}$, $\bold{i}'$ be two reduced words of $w_0$. It is
known that $\bold{i}'$ can be obtained from \textbf{i} through
several braid moves. Fix a path of braid moves from $\bf{i}$ to
$\bf{i}'$. For each move, there is a transform ( in Proposition 5.2,
\cite{K1}) between the Lusztig data of $P$ along the two consecutive
reduced words. By combining these transforms, we get a bijection
$R_{\bold{i}} ^{\bold{i}'}: \mathbb{N}^m\rightarrow \mathbb{N}^m$,
which is independent of the choice of the path from $\bold{i}$ to
$\bold{i}'$. We call it the \textsf{Lusztig transform} from
$\bold{i}$ to $\bold{i}'$ for $G$. From \cite{K1}, we also know that
$R^{\bold{i}'}_{\bold{i}}(n_\bullet)=n'_\bullet$ if and only if
$A^{\bold{i}}(n_\bullet)\cap A^{\bold{i}'}(n'_\bullet)$ is dense in
$A^\bold{i}(n_\bullet)$.

We give a necessary and sufficient condition on the
$\bold{i}$-Lusztig datum $n_\bullet$, so that $P$ is
$\st$-invariant. We call such $n_\bullet$ is a $\st$-invariant
$\bold{i}$-Lusztig datum.
\begin{prop}\label{prop:invariant}
Let $w_0=s_{\eta_1}s_{\eta_2}\cdots s_{\eta_m}$ be a reduced
expression of $w_0$ relative to the Coxeter group $W^\st$, where
$\eta_1, \eta_2, \cdots, \eta_m$, are orbits of $\st$ in $I$. For
each $\eta$, we fix a reduced expression of $s_\eta$ as an element
of $W$, and denote by $\bf{i}$ the resulting reduced expression of $
w_0$ relative to $W$. Let $n_\bullet$ be the $\bold{i}$-Lusztig
datum of $P$. Then $P$ is $\st$-invariant if and only if
$n_1=n_2=\cdots=n_{r_{\eta_1}}$,
$n_{r_{\eta_1}+1}=n_{r_{\eta_1}+2}=\cdots=n_{r_{\eta_1}+r_{\eta_2}},
\cdots$, where $r_\eta$ is the length of $s_\eta$ as an element of
$W$.

\end{prop}

\begin{proof}

For any orbit $\eta$ of $\st$, let $R_\eta$ be the root system
generated by $\{\alpha_i;i\in \eta\}$. Let $W_\eta$ be the Coxeter
group generated by \{$s_i$, for $i\in \eta$\}. Then $s_\eta$ is the
longest element in $W_\eta$.

Recall that $n_k$ means the length of the edge connecting
$\mu_{w^\bold{i} _{k-1}}$ with $\mu_{w^\bold{i}_k}$, i.e.
$\mu_{w^{\bold{i}}_k}-\mu_{w^{\bold{i}}_{k-1}}=-n_k.w^{\bold{i}}_{k-1}(\alpha^\vee
_{i_k})$. The convex hull of $\{\mu_w|w\in W_{\eta_1} \}$ forms an
MV polytope for an algebraic group of type $R_{\eta_1}$. We denote
it by $P^1 _{\eta_1}$. From $\mu_{w^\bold{i} _0},\cdots,
\mu_{w^\bold{i}_{r_{\eta_1}}}$, we get a Lusztig datum
$(n_1,n_2,\cdots,n_{r_{\eta_1}})$ along the chosen reduced word of
$s_\eta$. The convex hull of $\{\mu_w|w=s_{\eta_1}y, \text{ for }
y\in W_{\eta_2} \}$ forms an MV polytope of type $R_{\eta_2}$. We
denote it by $P^2 _{\eta_2}$. From $\mu_{w^\bold{i}_{r_{\eta_1}+1}
},\cdots, \mu_{w^\bold{i}_{r_{\eta_1}+r_{\eta_2}}}$, we get a
Lusztig datum
$(n_{r_{\eta_1}+1},n_{r_{\eta_1}+2},\cdots,n_{r_{\eta_1}+r_{\eta_2}})$
along the chosen reduced word of $s_{\eta_2}$. Similarly, we get
subsequently MV polytopes $P^3 _{\eta_3},\cdots, P^m _{\eta_m}$,
with type $R_{\eta_3},\cdots, R_{\eta_m}$. We also get their
corresponding Lusztig data along the chosen reduced words of
$s_{\eta_i}$.

Now let us return to the proof. If $P$ is $\st$-invariant, we have
$\st(\mu_w)=\mu_{\st(w)}$, for all $w\in W$, by Lemma \ref{equiv}.
Applying Lemma \ref{equiv} again, we see that $P^k _{\eta_k}$, for
all $k$, are $\st$-invariant.

Note that there are two possibilities: $A_2$ and $A_1\times
A_1\times \cdots \times A_1$ ( with $l$ copies of $A_1$, where
$\ell=$ 2 or 3) for $R_\eta$. Hence the sufficient part is reduced
to the following two cases which are easy to check.
\begin{enumerate}
\item
$A_2$,
 if $P$ is $\st$-invariant, then $n_1=n_2=n_3$ .
\item
$A_1\times A_1 \times\cdots \times A_1$, if $P$ is $\st$-invariant,
then $n_1=n_2=\cdots =n_l$.
\end{enumerate}

Conversely, from $A^\bold{i}(n_\bullet)=\cap_k S_{w_k ^\bold{i}}
^{\mu_{w_k ^{\bold{i}}}}$, we have
$\st(A^\bold{i}(n_\bullet))=A^\bold{j}(n_\bullet$), where
$\bold{j}=(\st(i_1),\st(i_2),\cdots,\st(i_m))$. From the condition
of $n_\bullet$, it is easy to see
$R^{\bold{j}}_{\bold{i}}(n_\bullet)=n_\bullet$. Hence their closures
coincide, i.e. the corresponding MV cycle of this $\bold{i}$-Lusztig
datum is $\st$-invariant. By Lemma \ref{equiv}, $P$ is
$\st$-invariant.
\end{proof}

\subsection{The bijection between MV cycles (polytopes) for a symmetric pair}\label{section 4.4}

Let $P$ be a $\st$-invariant MV polytope for $G$. In this
subsection, we will show that $P^\st$ is an MV polytope for $G^\st$,
and then we get the bijection between MV polytopes for a symmetric
pair.

Consider the symmetric pair $(A_4,B_2)$. For the longest element in
the Weyl group $W$, we have reduced expressions $w_0=s_1s_4\cdot
s_2s_3s_2\cdot s_1s_4\cdot s_2s_3s_2 =s_2s_3s_2\cdot s_1s_4\cdot
s_2s_3s_2 \cdot s_1s_4$. We get two reduced words $\bold{i}_\st$ and
$\bold{i}'_\st$ for $G^\st$ from these two expressions of $w_0$.
From $\bold{i}_\st$, and $\bold{i}'_\st$, we naturally get 2 reduced
words for $G$, $\bold{i}=(1,4,2,3,2,1,4,2,3,2)$,
$\bold{i}'=(2,3,2,1,4,2,3,2,1,4)$, respectively. Let $n_\bullet$,
$n'_\bullet$ be Lusztig data along $\bold{i}$, and $\bold{i}'$ for
$P$, respectively. According to Proposition \ref{prop:invariant}, we
may write $n_\bullet$ and $n'_\bullet$ as follows
\begin{eqnarray*}
&&n_\bullet=(\bar{n}_1,\bar{n}_1,\bar{n}_2,\bar{n}_2,\bar{n}_2,\bar{n}_3,\bar{n}_3,\bar{n}_4,\bar{n}_4,\bar{n}_4)\in \mathbb{N}^{10},\\
&&n'_\bullet=(\bar{n}'_1,\bar{n}'_1,\bar{n}'_1,\bar{n}'_2,\bar{n}'_2,\bar{n}'_3,\bar{n}'_3,\bar{n}'_3,\bar{n}'_4,\bar{n}'_4)\in
\mathbb{N}^{10},
\end{eqnarray*}
where $\bar{n}_k$, $\bar{n}'_k$ are non negative integers.

Set $n^\st_\bullet=(\bar{n}_1,\bar{n}_2,\bar{n}_3,\bar{n}_4)$. By
sending $n_\bullet$ to $n^\st_\bullet$, we get a bijection between
$\bold{i}$-Lusztig data of $\st$-invariant MV polytopes for $G$ and
$\bold{i}_\st$-Lusztig data of MV polytopes for $G^\st$. We shall
show this bijection is intrinsic, and independent of the choice of
reduced words. Note that the above procedure works for general case.

For any subvariety $Y \subset \G$, we set $Y^\st:=\{y\in Y|\st(y)=y
\}$.

Let $B(n_\bullet)=\{(b_\bullet)\in
\mathcal{K}^{\ell(w_0)}|\text{val}(b_k)=n_k, \forall\text{ } k \}$
and $B_\st(n^\st_\bullet)=\{(b_\bullet)\in
\mathcal{K}^{\ell_\st(w_0)}|\text{val}(b_k)=\bar{n}_k, \forall\text{
} k \}$, where val is the valuation function on $\mathcal{K}$.
Define a map $j_\st$ from $B_\st(n^\st_\bullet)$ to $B(n_\bullet)$,
by $
j_\st(b_1,b_2,b_3,b_4)=(b_1,b_1,b_2,2b_2,b_2,b_3,b_3,b_4,2b_4,b_4)$.

In this subsection, we always assume that $\bold{i}$ and $\bold{i'}$
are reduced words of $G$ resulting from the reduced words of
$G^\st$, $\bold{i}_\st$ and $\bold{i'}_\st$ respectively, in the
sense of Proposition \ref{prop:invariant}.

\begin{lem}\label{fixed}
Let $n_\bullet$ be a $\st$-invariant $\bold{i}$-Lusztig datum. Then
$A^\bold{i}(n_\bullet)^\st=A^{\bold{i}_\st}(n^\st_\bullet)$.
\end{lem}

\begin{proof}
We only show this lemma for the pair $(A_4,B_2)$, and the following
argument works in general.

Let $\iota :A^{\bold{i}_\st}(n^\st_\bullet)\hookrightarrow\G$ be the
natural imbedding, which is the restriction of $\iota: \G_\st
\hookrightarrow \G$. We have surjections
$\pi_{\bold{i}_\st}:B_\st(n^\st_\bullet)\rightarrow
A^{\bold{i}_\st}(n^\st_\bullet)$, and $\pi_\bold{i}: B(n_\bullet)
\rightarrow A^\bold{i}(n_\bullet)$, which are given by
\begin{equation*}
 \pi_{\bold{i}_\st}(b_1,b_2,b_3,b_4)=[\eta_{w_0} ^{-1}(
x_{\eta_1}(b_1)x_{\eta_2}(b_2)x_{\eta_1}(b_3)x_{\eta_2}(b_4))],
\end{equation*}
\begin{eqnarray*}
&\pi_\bold{i}(b_1,b_1,b_2,2b_2,b_2,b_3,b_3,b_4,2b_4,b_4)=\\
&[\eta_{w_0}^{-1}(x_1(b_1)x_4(b_1)\D x_2(b_2)x_3(2b_2)x_2(b_2)\D
x_1(b_3)x_4(b_3)\D x_2(b_4)x_3(2b_4)x_2(b_4))],
\end{eqnarray*}
where $x_{\eta_1}$ and $x_{\eta_2}$ are root subgroup homomorphisms
for $G^\st$, and we denote by $[\text{ }]$ the projection from
$G(\mathcal{K})$ to $\G$. For the definition of $\eta_{w_0}$, see
(section 4.4, \cite{K1}). Since $x_1(b_i)x_4(b_i)=x_{\eta_1}(b_i)$,
for i=1 or 3, and $x_2(b_j)x_3(2b_j)x_2(b_j)=x_{\eta_2}(b_j)$, for
j=2 or 4, we can see that $\iota\circ
\pi_{\bold{i}_\st}=\pi_\bold{i}\circ j_\st$, i.e., we have the
following commutative diagram
\[ \begin{CD}
@. B_\st(n^\st_\bullet) @>{j_\st}>>  B(n_\bullet)@. \\
@. @VV{\pi_{\bold{i}_\st}}V @VV{\pi_\bold{i}}V @.\\
@.A^{\bold{i}_\st}(n^\st_\bullet) @>{\iota}>>
  A^\bold{i}(n_\bullet).
\end{CD} \]

Since $\pi_{\bold{i}_\st}(
B_\st(n^\st_\bullet))=A^{\bold{i}_\st}(n^\st_\bullet)$, we have
$A^{\bold{i}_\st}(n^\st_\bullet)\subset A^\bold{i}(n_\bullet)^\st$.

Assume $n_\bullet$ is of coweight $\mu$. It is known that
$X(\mu)=S_e ^0 \cap S_{w_0} ^\mu=\bigsqcup
A^{\bold{i}}(n'_\bullet)$, where the union is taken over
$n'_\bullet$, such that $n'_\bullet$ is an $\bold{i'}$-Lusztig datum
with coweight $\mu$. Hence we have
\begin{equation}\label{equ1}
X(\mu)^\st=\bigsqcup A^\bold{i}(n'_\bullet)^\st,
\end{equation}
where $A^\bold{i}(n_\bullet)$ appear in the right hand side.

From Corollary \ref{cor}, we have the decomposition
\begin{equation}\label{equ2}
X(\mu)^\st=\bigsqcup A^{\bold{i}_\st}(m_\bullet),
\end{equation}
where the union is taken over $m_\bullet$ such that $m_\bullet$ is
an $\bold{i}_\st$-Lusztig datum with coweight $\mu$.

Let $m_\bullet=(m_1,m_2,m_3, m_4)$ be an $\bold{i}_\st$-Lusztig
datum, such that $\overline{A^\bold{i_\st}(m_\bullet)}$ is an MV
cycle for $G^\st$ with coweight $\mu$. Let
$n''_\bullet=(m_1,m_1,m_2,m_2,m_2,m_3,m_3,m_4,m_4,m_4)$. Then
$n''_\bullet$ is $\st$-invariant, and hence
$A^\bold{i_\st}(m_\bullet)\subset A^\bold{i}(n''_\bullet)^\st$. By
comparing decompositions of $X(\mu)^\st$ in (\ref{equ1}) and
(\ref{equ2}), we obtain
$A^\bold{i}(n_\bullet)^\st=A^{\bold{i}_\st}(n^\st_\bullet)$.
\end{proof}

\begin{rem}
From this lemma, we see that the closure of the fixed point set of
$\st$ on some open subset of a $\st$-invariant MV cycle $C$ is an MV
cycle for $G^\st$. We \textsf{believe} that the fixed point set of
$\st$ on $\st$-invariant MV cycle for $G$ is an MV cycle for $
G^\st$.
\end{rem}

\begin{cor}\label{nofixed}
  If $\overline{ A^\bold{i}(n_\bullet) }$ is not $\st$-invariant, then $A^\bold{i}(n_\bullet)^\st$ is
  empty.
\end{cor}

\begin{lem} \label{lem:nonempty}
If $n_\bullet$ is a $\st$-invariant $\bold{i}$-Lusztig datum, and
$R_\bold{i} ^{\bold{i}'} (n_\bullet)=n'_\bullet$, then
$(A^\bold{i}(n_\bullet)\cap A^{\bold{i}'}(n'_\bullet))^\st$ contains
an open dense subset.
\end{lem}
\begin{proof}
We can change $\bold{i}$ to $\bold{i}'$ by combining several braid
$d$-moves.

If $(\cdots,i_k,i_{k+1},i_{k+2},i_{k+3},\cdots) \mapsto
(\cdots,i_k,i_{k+2},i_{k+1},i_{k+3},\cdots)$, $(d=2)$, define a
rational map from $B(n_\bullet)$ to $B(n'_\bullet)$, by
\begin{equation*}
(\cdots,b_k,b_{k+1},b_{k+2},b_{k+3},\cdots)\mapsto(\cdots,b_k,b_{k+2},b_{k+1},b_{k+3},\cdots).
\end{equation*}

 If $(\cdots,i_k,i_{k+1},i_{k+2},i_{k+3},i_{k+4},\cdots) \mapsto
(\cdots,i_k,i_{k+2},i_{k+1},i_{k+2},i_{k+4},\cdots)$, $(d=3)$, where
$i_{k+1}=i_{k+3}$, then we define a rational map from $B(n_\bullet)$
to $B(n'_\bullet)$ by
\begin{eqnarray*}
(\cdots,b_k,b_{k+1},b_{k+2},b_{k+3},b_{k+4}\cdots)\mapsto(\cdots,
b_k,\frac{b_{k+2}b_{k+3}}{b_{k+1}+b_{k+3}},b_{k+1}+b_{k+3},\frac{b_{k+1}b_{k+2}}{b_{k+1}+b_{k+3}},b_{k+4},\cdots).
\end{eqnarray*}

It is well-known that, by several braid $d$-moves, we can arrive at
$\bold{i}'$ from $\bold{i}$. Let $\bold{i}\mapsto \bold{i}_1 \mapsto
\bold{i}_2 \mapsto \cdots \mapsto \bold{i}'$ be one such path, where
$\mapsto$ represents a braid $d$-move. For a path from $\bold{i}$ to
$\bold{i}'$, we denote the rational map $f$ by combining those in
every step defined above. Assume
$f(b_1,\cdots,b_m)=(b'_1,\cdots,b'_m)$. It is easy to see that
$b'_k$ is a rational function with numerator and denominator as
nonzero polynomials with nonnegative integral coefficients. Consider
the diagram
\begin{eqnarray*}
 && B(n_\bullet)\dashrightarrow B(n'_\bullet)\\
 && \downarrow \pi_\bold{i }\quad\quad\quad\quad \downarrow \pi_{\bold{i}'}\\
 && A^\bold{i}(n_\bullet)\dashrightarrow  A^{\bold{i}'}(n'_\bullet)
\end{eqnarray*}
where $\pi_\bold{i}$ is as in the proof of Lemma \ref{fixed}, and
dashed arrows denote rational maps. We have
$\pi_\bold{i}=\pi_{\bold{i}'}\circ f$.

Let $F$ be the product of all denominators appearing in every step
of $d$-moves, so it is a nonzero polynomial with nonnegative
integral coefficients. Let $U=\{(b_\bullet)\in
B(n_\bullet)|F(b_\bullet)\neq 0\}$. Then $f$ is well-defined on $U$,
and so $\pi_\bold{i}(U)\subset A^\bold{i}(n_\bullet)\cap
A^{\bold{i}'}(n'_\bullet)$.

There exists $y\in U$, such that $\pi_\bold{i}(y)\in
\pi_\bold{i}(U)\subset A^\bold{i}(n_\bullet)\cap
A^{\bold{i}'}(n'_\bullet)$, and $\pi_\bold{i}(y)$ is
$\st$-invariant. Hence $(A^\bold{i}(n_\bullet)\cap
A^{\bold{i}'}(n'_\bullet))^\st$ is nonempty. Since $\pi_\bold{i}$ is
an open map, $\pi_\bold{i}(U)$ is open in $A^\bold{i}(n_\bullet)$.
We  only show it in the case of $(A_4,B_2)$. Since
$\overline{A^\bold{i}(n_\bullet)}$ is $\st$-invariant, we have
$n_\bullet=(\bar{n}_1,\bar{n}_1,\bar{n}_2,\bar{n}_2,\bar{n}_2,\bar{n}_3,\bar{n}_3,\bar{n}_4,\bar{n}_4,\bar{n}_4)$.
Now take
$y=(t^{\bar{n}_1},t^{\bar{n}_1},t^{\bar{n}_2},2t^{\bar{n}_2},t^{\bar{n}_2},t^{\bar{n}_3},t^{\bar{n}_3},t^{\bar{n}_4},2t^{\bar{n}_4},t^{\bar{n}_4})\in
B(n_\bullet)$, then $F(y)\neq 0$. In the general case, we have the
similar argument.

Since $A^\bold{i}(n_\bullet)$ is irreducible by Lemma $\ref{fixed}$,
we have $(A^{\bold{i}}(n_\bullet)\cap
A^{\bold{i}'}(n'_\bullet))^\st$ is dense in
$A^{\bold{i}}(n_\bullet)^\st$.
\end{proof}


\begin{lem}\label{lem:pseudo}
Let $\text{\rm{Conv}}((\mu_w)_{w\in W^\st})$ be the convex hull of
$(\mu_w)_{w\in W^\st}$ in $t_\R$. If the MV polytope
$P=\text{\rm{Conv}}((\mu_w)_{w\in W})$ is $\st$-invariant, then
$P^\st=\text{\rm{Conv}}((\mu_w)_{w\in
  W^\st})$.
\end{lem}
\begin{proof}
Since $P$ is $\st$-invariant, we have $\st(\mu_w)=\mu_w$, for $w\in
W^\st$. We can easily see that $\st$ acts trivially on
$\text{\rm{Conv}}((\mu_w)_{w\in W^\st})$, so
$\text{\rm{Conv}}((\mu_w)_{w\in W^\st})\subset P^\st$.

For the converse. The perfect pairing $(X_*\otimes\mathbb{R}) \times
(X^*\otimes \mathbb{R}) \rightarrow \mathbb{R}$ descends to
$(X_*^\st\otimes \mathbb{R}) \times (X^*_\st\otimes
\mathbb{R})\rightarrow \mathbb{R}$ (see Section \ref{section 2.2}).
Note that $t_\mathbb{R}^\st$ can be identified with $X_*^\st\otimes
\mathbb{R}$.

For any $ \beta \in P^\st\subset P$, and $w\in W^\st$, we have
$\langle\beta,
w\D\lambda_i\rangle\leq\langle\mu_w,w\D\lambda_i\rangle$. By
descent, we have $\langle\beta, w\D\lambda_\eta\rangle
\leq\langle\mu_w,w\D\lambda_\eta\rangle$, for all orbit $\eta$ of
$\st$ in $I$, where $\lambda_\eta$ is the fundamental weight for
$G^\st$ corresponding to $\lambda_i$, for $i\in I$. Since
$P^\st\subset t_\mathbb{R}^\st$, we see that
$$P^\st\subset \{\beta \in t_\mathbb{R}^\st| \langle\beta,
w\D\lambda_\eta\rangle\leq \langle\mu_w,w\D\lambda_\eta\rangle,
\forall\text{ } \eta, \forall\text{ } w\in W^\st \}.$$

The right hand side is exactly $\text{\rm{Conv}}((\mu_w)_{w\in
W^\st})$.

\end{proof}

\begin{thm}\label{polytope}
If $P$ is a $\st$-invariant MV polytope for $G$, then $P^\st$ is an
MV polytope for $G^\st$.
\end{thm}

\begin{proof}
Let $\mu_\bullet$ be the vertices of $P$. Fix a reduced word
$\bold{i}_\st$ for $G^\st$, and let $n^\st_\bullet$ be the
corresponding $\bold{i}_\st$-Lusztig datum of $P$.

Let $\bold{i}$ be the fixed reduced word for $G$ from
$\bold{i}_\st$, in the sense of Proposition \ref{prop:invariant}.
Let $J=\{(\bold{i}', n'_\bullet)| \text{}\bold{i}'$ is a reduced
word for $G$ from some reduced word $\bold{i}'_\st$ for $G^\st$, and
$R^{\bold{i'}}_{\bold{i}}(n_\bullet)=n'_\bullet\}$. We have
$\cap_{(\bold{i}', n'_\bullet)\in J} A^{\bold{i}'}(n'_\bullet)^\st$
contains an open and dense subset of $A^\bold{i}(n_\bullet)^\st$
from Lemma \ref{lem:nonempty}, since the intersection of finite open
dense subsets is still open and dense.

Recall  $A^{\bold{i}'}(n'_\bullet)=\cap
S_{w^{\bold{i}'}_k}^{\mu_{w^{\bold{i}'}_k}}$, and
$A^{\bold{i}'_\st}(n'^\st_\bullet)=\cap
(S_\st)_{w^{\bold{i}'_\st}_k}^{\mu_{w^{\bold{i}'_\st}_k}}$. By Lemma
\ref{fixed}, we have $(\cap_{(\bold{i}', n'_\bullet)\in J}
A^{\bold{i}'}(n'_\bullet))^\st=\cap_{(\bold{i}'_\st,
n'^\st_\bullet)} A^{\bold{i}'_\st}(n'^\st_\bullet)=A((\mu_w)_{w\in
W^\st})$, where $A((\mu_w)_{w\in W^\st})=\cap_{w\in W^\st}
(S_\st)_{w} ^{\mu_w}$. The last equality holds, since for any $w\in
W^\st$, there exists some reduced word $\bold{i}'_\st$ of $G^\st$
and some integer $k$, such that $w=w^{\bold{i}'_\st}_k$. Therefore,
we have
$\overline{A^{\bold{i}_\st}(n^\st_\bullet)}=\overline{A^{\bold{i}}(n_\bullet)^\st}=\overline{(\cap_{(\bold{i}',
n'_\bullet)\in J}
A^{\bold{i}'}(n'_\bullet))^\st}=\overline{A((\mu_w)_{w\in W^\st})}$.
That means, the moment polytope of the MV cycle
$\overline{A^{\bold{i}_\st}(n^\st_\bullet)}$ is
$\text{\rm{Conv}}((\mu_w)_{w\in W^\st})$, which is exactly $P^\st$,
by Lemma \ref{lem:pseudo}. Hence $P^\st$ is really an MV polytope
for $G^\st$.

\end{proof}
\begin{cor}
Let $(\bold{i},n_\bullet)$ and $(\bold{i'},n'_\bullet)$ be two
$\st$-invariant Lusztig data. If
$R^{\bold{i}'}_{\bold{i}}(n_\bullet)=n'_\bullet$, then
$R^{\bold{i}'_\st}_{\bold{i}_\st}(n^\st_\bullet)=n'^\st_\bullet$
\end{cor}

\begin{thm} \label{main}
We have a bijection $\theta_P: \text{\rm{MVP}}_G ^\st
\longrightarrow \text{\rm{MVP}}_{G^\st}$, given by $P\mapsto P^\st$,
which preserves coweights. Induced from $\theta_P$, we have a
bijection $\theta_C:\text{\rm{MVC}} ^\st _G \longrightarrow
\text{\rm{MVC}}_{G^\st}$
\end{thm}

\begin{proof}
Let $P$ be a $\st$-invariant MV polytope for $G$. By Theorem
\ref{polytope}, we have a well-defined map $\theta_P:\text{\rm{MVP}}
^\st _G\longrightarrow \text{\rm{MVP}}_{G^\st}$ by
$\theta_P(P)=P^\st$.

 Fix a reduced word $\bold{i}_\st$ for $G^\st$. Let $\bold{i}$ be
a reduced word coming from $\bold{i}_\st$. For any MV polytope for
$G$ (resp. $G^\st$), we have the corresponding $\bold{i}$ (resp.
$\bold{i}_\st$) Lusztig datum. According to Proposition
\ref{prop:invariant}, $\theta_P$ is injective. Let $Q$ be any MV
polytope for $G^\st$, and let $m_\bullet$ be the
$\bold{i}_\st$-Lusztig datum of $Q$. By Lemma \ref{fixed} and its
proof, there exists a unique $\bold{i}$-Lusztig datum $n_\bullet$
such that $A^{\bold{i}_\st}(m_\bullet)$ is contained in
$A^{\bold{i}}(n_\bullet)$, and $n_\bullet$ is $\st$-invariant. Let
$P_Q$ be the MV polytope of $\overline{A^{\bold{i}}(n_\bullet)}$. We
have $P_Q ^\st=Q$, since $P_Q^\st$ has the same
$\bold{i}_\st$-Lusztig datum as $Q$. So $\theta_P$ is surjective.

Hence $\theta_P$ is a bijection, and it is easy to see that it
preserves the coweights of MV polytopes.
\end{proof}

\subsection{The bijection in highest weight case}\label{section 4.5}

Let $\lambda$, $\mu$ be $\st$-invariant coweights, we set
$X(\lambda, \mu):=S_e ^{\lambda} \cap S_{w_0} ^\mu$, and
$X(\mu-\lambda)=S_e ^0 \cap S_{w_0} ^{\mu-\lambda}$. In this
subsection, we have the same assumptions on the $\bold{i}$ and
$\bold{i}_\st$ as in Subsection \ref{section 4.4}.

The following lemma is given by Anderson\cite{A}
\begin{lem}\label{Anderson}
An irreducible component of $X(\lambda, \mu)$ is contained in
$\overline{\G ^\lambda}$ if and only if it appears as basis in
$V_\mu(\lambda)$
\end{lem}

First of all, we have a decomposition:
\begin{equation}\label{LD}
X(\lambda,\mu)=\lambda\cdot X(\mu-\lambda)=\bigsqcup \lambda\cdot
A^\bold{i}(n_\bullet),
\end{equation}
where the union is taken over $n_\bullet$ which are
$\bold{i}$-Lusztig data with coweight $\mu-\lambda$. Then
\begin{equation}\label{decomp0}
S_e^{\lambda} \cap S_{w_0}^{\mu}\cap \overline{\G
^\lambda}=\bigsqcup_1 \lambda\cdot A^\bold{i}(n_\bullet)  \cup
\bigsqcup_2 (\lambda\cdot A^\bold{i}(n_\bullet)\cap \overline{\G
^\lambda}),
\end{equation}
where the first union 1 is taken over those $n_\bullet$ in
(\ref{LD}) such that $\lambda\cdot A^\bold{i}(n_\bullet)\subset
\overline{\G ^\lambda}$; the second union 2 is taken over those
$n_\bullet$ in (\ref{LD}) such that $\lambda\cdot
A^\bold{i}(n_\bullet)\nsubseteq \overline{\G ^\lambda}$.

If $\lambda\cdot A^\bold{i}(n_\bullet)\nsubseteq \overline{\G
^\lambda}$, then $\lambda\cdot A^\bold{i}(n_\bullet)\cap
\overline{\G ^\lambda}$ is of lower dimension than
$A^\bold{i}(n_\bullet)$.

From decomposition (\ref{decomp0}) and Corollary \ref{nofixed}, we
have
\begin{equation}\label{decom1}
(S_e^{\lambda} \cap S_{w_0}^{\mu}\cap \overline{\G ^\lambda})^\st
=(S_e^{\lambda})^\st \cap (S_{w_0}^{\mu})^\st\cap (\overline{\G
^\lambda})^\st=\bigsqcup_3 \lambda\cdot A^\bold{i}(n_\bullet)^\st
\cup \bigsqcup_4 (\lambda\cdot A^\bold{i}(n_\bullet)\cap
\overline{\G ^\lambda})^\st,
\end{equation}
where the first union 3 is taken over those $n_\bullet$ in
(\ref{LD}), such that $\lambda\cdot A^\bold{i}(n_\bullet) \subset
\overline{\G ^\lambda}$ and $n_\bullet$ is $\st$-invariant; the
second union 4 is taken over those $n_\bullet$ in (\ref{LD}), such
that $\lambda\D A^\bold{i}(n_\bullet)\nsubseteq \overline{\G
^\lambda}$ and $n_\bullet$ is $\st$-invariant. From the point view
of $G^\st$, we also have a decomposition

\begin{equation}\label{decomp2}
(S_\st)_e^{\lambda} \cap (S_\st)_{w_0}^{\mu} \cap (\overline{\G
^\lambda})^\st=\bigsqcup_5 \lambda \cdot A^{\bold{i}_\st}(m_\bullet)
\cup \bigsqcup_6 (\lambda\cdot A^{\bold{i}_\st}(m_\bullet)\cap
\overline{\G_\st ^\lambda}),
\end{equation}
 where the
first union 5 is taken over $m_\bullet$ which are
$\bold{i}_\st$-Lusztig data with coweight $\mu-\lambda$, satisfying
$\lambda\cdot A^{\bold{i}_\st}(m_\bullet)\subset \overline{\G_\st
^\lambda}$; the second union 6 is taken over $m_\bullet$ which are
$\bold{i}_\st$-Lusztig data with coweight $\mu-\lambda$, satisfying
$\lambda\cdot A^{\bold{i}_\st}(m_\bullet)\nsubseteq \overline{\G_\st
^\lambda}$.

If $\lambda\cdot A^{\bold{i}_\st}(m_\bullet)\nsubseteq \overline{\G
^\lambda}$, then $\lambda\cdot A^{\bold{i}_\st}(m_\bullet)\cap
\overline{\G_\st ^\lambda}$ is of lower dimension than
$A^{\bold{i}_\st}(m_\bullet)$.

\begin{lem}
  $\overline{\G ^\lambda}=\overline{\cap S_w ^{w\cdot\lambda}}$.
\end{lem}
\begin{proof}
 We know $\overline{\cap S_w ^{w\cdot\lambda}}$ is an MV cycle with coweight $(\lambda,w_0\D\lambda)$, and it is contained
 in  $\overline{\G ^\lambda}$. Since both of them are of the same dimension
 $2\langle \lambda,\rho\rangle$, and both of them are irreducible, we have $\overline{\G ^\lambda}=\overline{\cap S_w
 ^{w.\lambda}}$.
\end{proof}

\begin{lem}\label{lem:dim}
If  $\lambda\cdot A^\bold{i}(n_\bullet)\nsubseteq \overline{\G
^\lambda}$, and $n_\bullet$ is $\st$-invariant, then $(\lambda\cdot
A^\bold{i}(n_\bullet)\cap \overline{\G ^\lambda})^\st$ is of lower
dimension than $A^\bold{i}(n_\bullet)^\st$.
\end{lem}
\begin{proof}
With the same reason as in the proof of lemma \ref{fixed}, we can
find an open subset $U\subset B(n_\bullet)$, such that
$\pi_\bold{i}(U)\subset
\cap_{(\bold{i},n_\bullet)}A^\bold{i}(n_\bullet)=\cap_w S_w
^{\mu_w}$ is open in $A^\bold{i}(n_\bullet)$.

 Note that $(\cap \lambda\cdot S_w ^{\mu_w}) \cap \overline{\G^\lambda}$ is empty.
Otherwise, if there exists a point $p\in (\cap \lambda\cdot S_w
^{\mu_w} )\cap \overline{\G^\lambda}$, then
\begin{equation*}
p \in  (\cap \lambda\cdot S_w ^{\mu_w}) \cap \overline{\G^\lambda}
=(\cap \lambda\cdot S_w ^{\mu_w}) \cap \overline{\cap S_w
^{w\cdot\lambda}} \subset (\cap \lambda\cdot S_w ^{\mu_w}) \cap
\overline{ S_w ^{w\cdot\lambda}}.
\end{equation*}
That is, $\forall\text{ } w\in W$, $p$ must be contained in
$\lambda\cdot S_w ^{\mu_w} \cap \overline{S_w ^{w\cdot\lambda}}$.
From $\overline{S_w ^{w\cdot\lambda}}=\bigsqcup_{\mu\leq_w
w\cdot\lambda} S_w ^\mu$, we have $\mu_w+\lambda \leq_w
w\cdot\lambda$. We get that
$\text{\rm{Conv}}(\mu_\bullet)+\lambda\subset
\text{\rm{Conv}}(W\cdot\lambda)$. According to Anderson's theorem on
multiplicity of weight space \cite{A}, we have $\lambda\cdot
\overline{A(\mu_\bullet)}$ is an MV cycle in $V_\mu(\lambda)$. By
Lemma \ref{Anderson}, it is a contradiction to the condition that
$\lambda\cdot A^\bold{i}(n_\bullet)\nsubseteq \overline{\G
^\lambda}$. As in Lemma $\ref{lem:nonempty}$, there exists a point
$p\in$ $\lambda\cdot A^\bold{i}(n_\bullet)$. So $\lambda\cdot
A^\bold{i}(n_\bullet)^\st \cap \overline{\G ^\lambda}^\st$ has lower
dimension than $A^\bold{i}(n_\bullet)^\st$.
\end{proof}

By Lemma \ref{lem:dim}, and by comparing the two decompositions
(\ref{decom1}) and (\ref{decomp2}), we have that the set
$\{A^\bold{i}(n_\bullet)| n_\bullet \text{ is } \st
\text{-invariant}\text{ and is of coweight } \mu-\lambda, \text{ and
} \lambda\cdot A^\bold{i}(n_\bullet)\subseteq \overline{\G ^\lambda}
\}$ is in bijection with the set
$\{A^{\bold{i}_\st}(m_\bullet)|m_\bullet\text{ is of coweight }
\mu-\lambda,\text{ and } \lambda\cdot
A^{\bold{i}_\st}(m_\bullet)\subseteq \overline{\G_\st ^\lambda} \}$,
 by sending $A^\bold{i}(n_\bullet)$ to $A^\bold{i}(n_\bullet)^\st$.
We thus obtain the following theorem.
\begin{thm}\label{highest}
We have a bijection $\theta_C^\lambda:
\text{\rm{MVC}}_G(\lambda)^\st\longrightarrow\text{\rm{MVC}}_{
G_\st}(\lambda)$, which is the restriction of $\theta_C$ in Theorem
\ref{main}.
\end{thm}

\section{Twining character formula}\label{section 5}
Recall that $\text{\rm{Perv}}_{G(\mathcal{O})}(\G)$ is a tensor
category \cite{MV}, and it is easy to see the tensor functor $\st^*$
induced from the action of $\st$ on affine Grassmannian is a tensor
equivalence. From the functoriality of Tannakian formalism
\cite{DM}, we have a natural automorphism $\bar{\st}$ on $G^\vee$.

Fix a $\st$-invariant coweight $\lambda$, and choose an isomorphism
$\phi: IC_\lambda\simeq \st^*(IC_\lambda)$, which is compatible with
the action of $\st$ on MV cycles( as the basis of $V(\lambda)$).
\begin{lem}\label{Tanna}
The action of $\bar{\st}$ on $G^\vee$ is compatible with the natural
action of $\st$ on $V(\lambda)$ induced from $\phi$.
\end{lem}
\begin{proof}  Let $T$ be the functor from $\text{\rm{Perv}}_{G(\mathcal
{O})}(\G)$ to $\text{\rm{Rep}}(G^\vee)$, such that
$T(IC_\lambda)=(\rho_\lambda, V(\lambda))$, where
$\rho_\lambda:G^\vee\rightarrow GL(V(\lambda))$ is the corresponding
representation.

From $\st^*:\text{\rm{Perv}}_{G(\mathcal {O})}(\G)\rightarrow
\text{\rm{Perv}}_{G(\mathcal {O})}(\G)$, we get
$T(\st^*(IC_\lambda))=(\rho_\lambda\circ \bar{\st}, V(\lambda))$.
Let $\tilde{\st}$ be the functor from $\text{\rm{Rep}}(G^\vee)$ to
$\text{\rm{Rep}}(G^\vee)$, by sending $(\rho_\lambda, V(\lambda))$
to $(\rho_\lambda\circ \bar{\st}, V(\lambda))$. Then we have the
following commutative diagram:
\[ \begin{CD}
@. \text{\rm{ Perv}}_{G(\mathcal {O})}(\G)@>{T}>> \text{\rm{Rep}}(G^\vee) @. \\
@. @VV{\st^*}V @VV{\tilde{\st}}V @.\\
@. \text{\rm{ Perv}}_{G(\mathcal
{O})}(\G)@>{T}>>\text{\rm{Rep}}(G^\vee)
\end{CD} \]
By applying $T$ to $\phi: IC_\lambda\simeq \st^*(IC_\lambda)$, we
obtain an isomorphism $\st=T(\phi): (\rho_\lambda,
V(\lambda))\rightarrow (\rho_\lambda\circ \bar{\st}, V(\lambda))$ in
$\text{\rm{Rep}}(G^\vee)$. In other words, there exists a linear
isomorphism $\st:V(\lambda)\rightarrow V(\lambda)$ satisfying
\begin{equation*}
\st(\rho_\lambda(g)\D v)=(\rho_\lambda\circ \bar{\st})(g)\D
\st(v)=\rho_\lambda(\bar{\st}(g))\D \st(v), (g\in G^\vee, v\in
V(\lambda)).
\end{equation*}
\end{proof}

\begin{thm}
$\bar{\st}$ is a Dynkin automorphism on $G^\vee$.
\end{thm}
\begin{proof}
Let $\text{\rm{Vect}}_{X_*}$ be the tensor category of $X_*$-graded
vector spaces. The action of $\st$ on $X_*$ induces an tensor
functor $\st^\circ$ on $\text{\rm{Vect}}_{X_*}$. From
Mirkovic-Vilonen's paper \cite{MV}, we know that there is a tensor
functor $F$ from $\text{\rm{Perv}}_{G(\mathcal{O})}(\G)$ to
$\text{\rm{Vect}}_{X_*}$, and it's easy to see $\st^*$ and
$\st^\circ$ are compatible with $F$.

Applying Tannkian formalism, from $F$ we get the forgetful functor
from $\text{\rm{Rep}}(G^\vee)$ to $\text{\rm{Rep}}(T^\vee)$, where
$T^\vee$ is a torus of $G^\vee$, and $\st^*$, $\st^\circ$ induce
automorphisms on $G^\vee$ and $T^\vee$, respectively. Since $\st^*$
and $\st^\circ$ are compatible with $F$, we have $\bar{\st}$
preserve the torus $T^\vee$, i.e, $\bar{\st}(T^\vee)=T^\vee$. It
induces the action of $\st$ on $X^*(T^\vee)$.

Let $B^\vee$ be the maximal subgroup of $G^\vee$, which stabilizes
the highest weight line $V_{\lambda}(\lambda)$ in $V(\lambda)$, for
every $\st$-invariant dominant weight $\lambda$. It's easy to see
$B^\vee$ is a Borel subgroup of $G$, and contains $T^\vee$;
furthermore, $\st(B^\vee)=B^\vee$.

The coroots of $G$ $\alpha^\vee_i$, $i\in I$, can be viewed as the
roots of $G^\vee$, and $\st$ send the root $\alpha^\vee_i$ to
$\alpha^\vee_{\st(i)}$ automatically, since under the identification
of $X^*(T^\vee)$ and $X_*$, the actions of $\st$ are compatible.

Let $\mathscr{G}^\vee$ be the Lie algebra of $G^\vee$. Let $\tau$ be
the automorphism on $\mathscr{G}^\vee$ induced from $\bar{\st}$.
From the following Lemma \ref{tau}, we know $\tau$ acts trivially on
the simple root space $\mathscr{G}^\vee_{\alpha^\vee_i}$, for $i$
fixed by $\st$. Lift $\tau$ to $\bar{\st}$ on $G^\vee$, then
$\bar{\st}$ act trivially on the root subgroup $U_{\alpha^\vee_i}$
and $U_{-\alpha^\vee_i}$, for $i$, $\st(i)=i$. Hence we are able to
find root subgroup homomorphisms $x^\vee_i: \C\rightarrow G$ and
$y^\vee_i:\C\rightarrow G$, corresponding to $\alpha^\vee_i$ and
$-\alpha^\vee_i$, such that $\st(x^\vee_i(a))=x^\vee_{\st(i)}(a)$
and $\st(y^\vee_i(a))=y^\vee_{\st(i)}(a)$ for all $i\in I$.

Hence  $\st$ is a Dynkin automorphism with respect to a pinning of
$G^\vee$, $(G,T,B,x^\vee_i,y^\vee_i, i\in I)$ .
\end{proof}
Assume the highest root is $\gamma^\vee$, then it is
$\st$-invariant. $\mathscr{G^\vee}$ admits a highest representation
of $G^\vee$ with highest weight $\gamma^\vee$. Assume
$e_{\alpha^\vee}$ is the basis corresponding to the unique MV cycle
in the root space $\mathscr{G}^\vee_{\alpha^\vee}$. By interchanging
MV cycles, we get a linear operator $\st$ on $\mathscr{G}^\vee$,
especially $\st(e_{\alpha^\vee})=e_{\st(\alpha^\vee)}$. Recall
$\tau$ is an automorphism on $\mathscr{G}^\vee$, we have

\begin{lem}{\label{tau}}
As linear operators on $\mathscr{G}^\vee$, if $G^\vee$ is of type
$A_{2n}$, then $\tau=-\st$; otherwise $\tau=\st$.
\end{lem}
\begin{proof}
Let $\mathscr{H}^\vee$ be the Lie algebra of $T^\vee$. It is a
Cartan subalgebra of $\mathscr{G}^\vee$, and it can be identified
with $X^*\otimes \C$, where the actions of $\tau$ on
$\mathscr{H}^\vee$ and $\st$ on $X^*$ are compatible.

From Lemma \ref{Tanna}, we have $\st([a,b])=[\tau(a), \st(b)]$, for
two arbitrary elements $a$ and $b$ in $\mathscr{G}^\vee$. By Schur's
lemma, we have $\tau=c\cdot \st$, for some constant $c$. Let
$\gamma$ be the corresponding coroot of highest root $\gamma^\vee$,
so it is $\st$-invariant. Since $[e_{\gamma^\vee},
e_{-\gamma^\vee}]\in \C\cdot \gamma$, we have $[e_{\gamma^\vee},
e_{-\gamma^\vee}]=\tau([e_{\gamma^\vee},
e_{-\gamma^\vee}])=[\tau(e_{\gamma^\vee}),\tau(e_{-\gamma^\vee})]=c^2\cdot
[e_{\gamma^\vee}, e_{-\gamma^\vee}]$. Hence $c^2=1$.

If $G^\vee$ is of type $A_{2n}$. There exists two adjacent simple
roots $\alpha^\vee_i$ and $\alpha^\vee_j$, such that $\st(i)=j$, for
$i$ and $j \in I$. Then we have
$\tau([e_{\alpha^\vee_i},e_{\alpha^\vee_j}])=[e_{\alpha^\vee_j},e_{\alpha^\vee_i}]=-[e_{\alpha^\vee_i},e_{\alpha^\vee_j}]$.
Since $\alpha^\vee_i+\alpha^\vee_j$ is also $\st$-invariant, it
forces $c=-1$.

If $G^\vee$ is of other type. Let $h_i=[e_{\alpha^\vee_i},
e_{-\alpha^\vee_i}]$. Since $\st([e_{\alpha^\vee_i}$,
$e_{-\alpha^\vee_i}])=[\tau(e_{\alpha^\vee_i})$,
$\st(e_{-\alpha^\vee_i})]=c\cdot [e_{\alpha^\vee_{\st(i)}}$,
$e_{\alpha^\vee_{-\st(i)}}]$, we have $\st(h_i)=c\cdot h_{\st(i)}$.
Then $\{h_i\}_{i\in I}$ is a basis of $\mathscr{H}^\vee$. Since
there exists $i\in I$, such that $\st(i)=i$, when $G^\vee$ is not of
type $A_{2n}$, it's easy to see $\text{\rm{trace}}(
\st|_{\mathscr{H}^\vee})>0$. Moreover, $\st$ interchanges MV cycles
in $\mathscr{H^\vee}$, so $\text{\rm{trace}}(
\tau|_{\mathscr{H}^\vee})\geq 0$. We thus have $c=1$.

 \end{proof}

\begin{rem} We can give another construction of Dynkin
automorphism on $G^\vee$ which is compatible with the action of
$\st$ on MV cycles, by using Vasserot's explicit construction of the
action of dual group on cohomology of perverse sheaves \cite{V}.
Moreover, this automorphism coincides with the one from Tannakian
formalism.
\end{rem}

Recall that twining character is defined to be
$\text{\rm{ch}}^\st(V(\lambda)):=\sum_{\mu\in P(\lambda)^\st}
\text{\rm{trace}}(\st |_{V_\mu(\lambda)})e^\mu$ for a Dynkin
automorphism $\st$, where $\lambda$ is $\st$-invariant.

\begin{prop}
\begin{equation*}
\text{\rm{ch}}^\st (V(\lambda))=\frac{\sum_{w\in W^\st}
(-1)^{\ell_\st (w)}e^{w(\lambda+\rho)}}{\sum_{w\in W^\st}
(-1)^{\ell_\st (w)}e^{w(\rho)}}.
\end{equation*}
\end{prop}

\begin{proof}
Let $V^\st (\lambda)$ be the irreducible representation of
$(G^\st)^\vee$ with highest weight $\lambda$. By Weyl character
formula for $G^\st$,  we have $ \sum_{\mu\in P(\lambda)^\st} \dim
V^\st_\mu(\lambda)e^{\mu}=\frac{\sum_{w\in W^\st} (-1)^{\ell_\st
(w)}e^{w(\lambda+\rho)}}{\sum_{w\in W^\st} (-1)^{\ell_\st
(w)}e^{w(\rho)}}$.

Comparing with our definition of twining character for $G$, we see
that it is equivalent to show
$\text{\rm{trace}}(\st|_{V_\mu(\lambda)})=\dim V^\st_\mu(\lambda)$,
for any $\mu\in P(\lambda)^\st$. By Lemma \ref{Tanna},
$\text{\rm{trace}}(\st|_{V_\mu(\lambda)}=\sharp(\text{\rm{MVC}}_G
^\mu (\lambda) ^\st)$. Hence our proposition follows from Theorem
\ref{highest}
\end{proof}

\end{document}